\documentclass[reqno]{amsart}
\usepackage{hyperref}

\begin{document}
\title[\hfilneg  \hfil Nonlinear Love-equation with infinite memory]
{General decay rate of solution for nonlinear Love-equation with infinite memory}

\author[Khaled Zennir \hfil 11-2018\hfilneg]
{Khaled Zennir}

 \address{Khaled zennir \newline
 	First address: Laboratory LAMAHIS, University 20 Ao\^ut 1955- Skikda, 21000, Algeria\newline
 	Second address:  Department of Mathematics, College of
 	Sciences and Arts, Al-Ras, Qassim University, Kingdom of Saudi
 	Arabia.}
 \email{khaledzennir4@yahoo.vom}

\thanks{ }

\begin{abstract}
 In this paper, we consider a nonlinear Love-equation with infinite memory. By certain properties of convex functions, we use an appropriate Lyapunov functional to find a very general rate of decay for energy (\ref{Energy}).
\end{abstract}

\maketitle
\numberwithin{equation}{section}
\newtheorem{theorem}{Theorem}[section]
\newtheorem{Proposition}[theorem]{Proposition}
\newtheorem{Definition}[theorem]{Definition}
\newtheorem{lemma}[theorem]{Lemma}
\newtheorem{remark}[theorem]{Remark}
\newtheorem{example}[theorem]{Example}
\allowdisplaybreaks
\section{Introduction}
Denote $y=y(x,t)$, $y'=y_{t}=\frac{\partial y}{\partial t}
(x,t), y''=y_{tt}=\frac{\partial^2y}{\partial t^2}(x,t), y_{x}=\frac{\partial y}{\partial x}(x,t), y_{xx}=\frac{\partial^2y}{\partial x^2}(x,t), x\in \Omega =(0,L), L>0, t>0$. In this article, we consider a nonlinear Love-equation in the form
\begin{gather}
\begin{aligned}
& y''-  \big(y_{x}+y'_x+y''_{x}\big)_x  + \int_{-\infty}^{t}\mu(t-s)y_{xx}(s)ds  \\
&=F[y]  -\Big(F[y] \Big)_x    +f(x,t),\quad x\in \Omega,\; 0<t<T,
\end{aligned} \label{1}
\end{gather}
where
\begin{equation} \label{b5}
F[y] = F\Big( x, t, y, y_{x}, y', y_{x}'\Big) \in C^{1}\Big([0,1]\times\mathbb{R}^+\times\mathbb{R}^4\Big),
\end{equation}
The given functions $\mu, f$ are specified later.
With $F=F(x,t,y_1,\dots,y_4)$, we put
$D_1F=\frac{\partial F}{\partial x}$,
$D_2F=\frac{\partial F}{\partial t}$,
$D_{i+2}F=\frac{\partial F}{\partial y_i}$, with $i=1,\dots,4$.\\
Equation (\ref{1}) satisfies the homogeneous Dirichlet boundary conditions:
\begin{equation}\label{2}
y(0,t)=y(L,t)=0, \qquad t>0,
\end{equation}
and the following initial conditions
\begin{equation}\label{3}
y(x,-t)=y_0(x,t),\quad y'(x,0)=y_1(x).
\end{equation}
To deal with a wave equation with infinite history, we assume that the kernel function $\mu$ satisfies the following hypothesis:\\
(Hyp1:) $\mu: \mathbb{R}^+ \rightarrow \mathbb{R}^+$ is a non-increasing $C^1$ function such that
\begin{eqnarray}
1- \int_{0}^{\infty}\mu(s)ds=l>0, \quad \mu(0)>0. \label{Cond2g}
\end{eqnarray}
and there exists an increasing strictly convex function $H: \mathbb{R}^+ \rightarrow \mathbb{R}^+$ of $C^1(\mathbb{R}^+)$ satisfying
\begin{eqnarray}
H(0)=H'(0)=0,\quad and \quad \lim\limits_{t \rightarrow \infty} H'(t)=\infty.
\end{eqnarray}
such that
\begin{eqnarray}
\int_{0}^{\infty}\frac{\mu(s)}{H^{-1}(-\mu'(s))}ds+\sup_{s\in \mathbb{R}^+}\frac{\mu(s)}{H^{-1}(-\mu'(s))}ds<\infty \label{H}
\end{eqnarray}
\begin{itemize}
	\item[(Hyp2:)]  $y_0(0), y_1\in H_0^1\cap H^2$; 
\end{itemize}
We need the following assumptions on source forces:\\
\begin{itemize} 
	\item[(Hyp3:)]  $f\in H^1((0,1)\times (0,T))$;
	 	
	\item[(Hyp4:)]  $F\in C^1\Big([0,1]\times [ 0,T]\times\mathbb{R}^4\Big)$, such that $F(0,t,0, y_2,0,y_4)=F(1,t,0, y_2,0,y_4)=0$ for all $t\in [0,T], \qquad y_2,y_4\in\mathbb{R}$.
\end{itemize}
This kind of systems appears in the models of nonlinear Love waves or Love type waves. It is a generalization of a model introduced by \cite{Love1972}, \cite{Love1964} and \cite{Love1978}. The original equation is
\begin{eqnarray}
u''-\frac{E}{\rho}u_{xx}-2\mu^2w^2u''_{xx}=0, \label{1.4}
\end{eqnarray}
This type of problem describes the vertical oscillations of a rod, was established from Euler's variational equation of an energy functional associated with (\ref{1.4}). A classical solution of problem (\ref{1.4}) with null boundary conditions and asymptotic behavior are obtained by using the Fourier method and method of small parameter. The results are very  interesting in the application point of view and, ad for as I know, new that is there is no results for equations of Love waves or Love type waves with the presence of finite/infinite memory term (\cite{Love2009}, \cite{Love1972}, \cite{Love2013}, \cite{Love2014}, \cite{Love2015}, \cite{Love2016}, \cite{Love1964}, \cite{Love1978}, \cite{Love2017}, \cite{Love2018}, \dots).\\
Without infinite memory term, when $\lambda=0$ in (\ref{1}), Triet and his collaborator in \cite{Love2016} considered an IBVP for a nonlinear Kirchhoff-Love equation
\begin{gather}
	\begin{aligned}
		&u_{tt}-\frac{\partial }{\partial x}   
		\big(u_{x}+\lambda _1u_{xt}+u_{xtt}\big)  +\lambda u_{t}  =F\big( x,t,u,u_{x},u_{t},u_{xt} \big)    \\
		&\quad-\frac{\partial }{\partial x}\big[ G\big(
		x,t,u,u_{x},u_{t},u_{xt} \big) \big]     +f(x,t),\quad x\in \Omega =(0,1),\; 0<t<T,
	\end{aligned}   \\
	u(0,t)=u(1,t)=0,   \\
	u(x,0)=\tilde{u}_0(x),\quad u_{t}(x,0)=\tilde{u}_1(x),   
\end{gather}
where $\lambda >0$, $\lambda _1>0$ are constants and
$\tilde{u}_0, \tilde{u}_1\in H_0^1\cap H^2;$ $f$, $F$ and $G$ are given functions. First, under suitable conditions, the existence of a unique local weak solution has been proved and a blow up result for solutions with negative initial energy is also established. A sufficient condition guaranteeing the global existence and exponential decay of
weak solutions is given in the last section. This results will be improved in \cite{Love2017}, \cite{Love2018} to the Kirchhoff typ.\\
The existence/ nonexistence, exponential decay of solutions and blow up results for viscoelastic wave equations with finite history, have been extensively studied and many results have been obtained by many authors (see \cite{FFV2009}, \cite{FV1968}, \cite{FV2001}, \cite{FV2009}, \cite{FV2012}, \cite{FV2003}, \cite{FV2008}, \cite{FV2015}, \cite{FV2017} \dots).\\
Concerning problems with infinite history, we mention the work \cite{FV2001} in which considered the following semi-linear hyperbolic equation, in a bounded domain  of $ \mathbb{R}^3$,
$$
u''-K(0)\Delta u-\int_{0}^{\infty}K'(s)\Delta u(t-s)ds+g(u)=f,
$$
with $K(0), K(\infty)>0, K'\leq 0$ and gave the existence of global attractors for the problem. Next in \cite{FG2017}, the authors considered a fourth-order suspension bridge equation with nonlinear damping term and source term.  The authors found necessary and sufficient condition for global existence and energy decay results without considering the relation between $m$ and $p$. Moreover, when   $p>m$, they gave sufficient condition for finite time blow-up of solutions. The lower bound of the blow-up time is also established.\\
Recently, in \cite{IIV2018}, the authors studied a three-dimensional (3D) visco-elastic wave equation with nonlinear weak damping, supercritical sources and prescribed past history, $t\leq0$ in
$$
u''-k(0)\Delta u-\int_{0}^{\infty}k'(s)\Delta u(t-s)ds+\vert u'\vert^{m-1}u'=\vert u\vert^{p-1}u,
$$
where the relaxation function $k$ is monotone decreasing with $k(+\infty)=1, m\geq1, 1\leq p<6.$ When the source is stronger than dissipations, i.e. $p>\max\{m, \sqrt{k(0)}\}$, they obtained some finite time blow-up results with positive initial energy. In particular, they obtained the existence of certain solutions which blow up in finite time for initial data at arbitrary energy level. (see \cite{IV2012}, \cite{IV2013}, \cite{IV2018}, \cite{IIV2018}, \cite{IV2017}, \dots).\\
In addition to the existence results in Theorem \ref{thm2.2} and Theorem \ref{global theorem} obtained by a new combined methods, our decay rate, which is very general, obtained in  the fourth section, Theorem \ref{decay theorem} extend that obtained in \cite{FV2015} and \cite{IV2012}, where they established a general decay rate result for relaxation functions satisfying 
\begin{equation}
g'(t)\leq -H(g(t)), t\geq 0,\quad H(0)=0
\end{equation} 
for a positive function $H \in C^{1}(\mathbb{R}^{+})$ and $H$ is linear or strictly increasing and strictly convex $C^{2}$ function on $(0,r], 1>r$. This improves the conditions introduced by \cite{FFV2009} on the relaxation functions
\begin{eqnarray}
g'(t)\leq -\chi(g(t)),\quad \chi(0)=\chi'(0)=0
\end{eqnarray}
where $\chi$ is a non-negative function, strictly increasing and strictly convex on $(0,k_{0}], k_{0}>0$. They required that
\begin{eqnarray}
\int_{0}^{k_{0}}\frac{dx}{\chi(x)} =+\infty,\quad \int_{0}^{k_{0}}\frac{xdx}{\chi(x)} <1,\quad \lim\limits_{s\rightarrow 0^{+}} \inf\frac{\chi(s)/s}{\chi'(s)}>\frac{1}{2}
\end{eqnarray}
and proved a decay result for the energy in a bounded domain. In addition to these assumptions, if 
\begin{eqnarray}\lim\limits_{s\rightarrow 0^{+}} \sup\frac{\chi(s)/s}{\chi'(s)}<1,
\end{eqnarray}
then, in this case, an explicit rate of decay is given.
\section{The existence of solution} 
We define the weak solution to of  \eqref{1}--\eqref{3} as follows.
\begin{Definition}
	A function $y$ is said to be a weak solution of \eqref{1}--\eqref{3} on $[0,T]$ if  $$ y, y', y''\in L^{\infty }( 0,T ;H_0^1\cap H^2),$$ such that $y$
	satisfies the variational equation
	\begin{equation}
	\begin{aligned}
	&\int_{\Omega} y'' w\  dx +\int_{\Omega} ( 
	y_{x} +y_{x}' +y_{x}'' ) w_{x}dx\\
	&-\int_{\Omega}\int_{0}^{\infty}\mu(s)y_{x}(t-s)dsw_xdx  \\
	&=\int_{\Omega} f wdx +\int_{\Omega}
	F[y] wdx +\int_{\Omega} F[y] w_{x}dx,
	\end{aligned}  \label{b3}
	\end{equation}
	for all test function $w\in H_0^1$, for almost all $t\in (0,T)$.
\end{Definition}	
The following famous and widely used technical lemma will play an important role in the sequel.
\begin{lemma}\label{lemg} [\cite{IVLove2018}]
	For any $v\in C^{1}\left( 0,T,H^{1}_0\right) $\ we have%
	\begin{eqnarray}
	&&\int_{\Omega} \int_{0}^{\infty}\mu(s)v_{xx}(t-s)v'(t)dsdx\nonumber\\
	&=&\frac{1}{2}\frac{d}{dt} \int_{0}^{\infty}\mu(s)\int_{\Omega}\vert
	v(t-s)-v(t)\vert ^2dxds-\frac{1}{2}\frac{d}{dt} \int_{0}^{\infty}\mu(s)ds\int_{\Omega}\left\vert v_x(t)\right\vert ^{2}dx \nonumber \\
	&&-\frac{1}{2} \int_{0}^{\infty}\mu'(s)\int_{\Omega}\vert
	v(t-s)-v(t)\vert ^2dxds. \nonumber
	\end{eqnarray}
\end{lemma}
Now, we state the existence of a local solution for \eqref{1}--\eqref{3}. 
\begin{theorem} \label{thm2.2} [\cite{IVLove2018}, Theorem 2.3]
Let $y_0(0), y_1\in H_0^1\cap H^2$ be given. Assume that {\rm(Hyp1)--(Hyp4)} hold. Then Problem
\eqref{1}--\eqref{3}  has a unique local solution $y$ and
\begin{equation}
\begin{gathered}
y, y', y''\in L^{\infty }( 0,T_{\ast };H_0^1\cap H^2),
\end{gathered}  \label{b6}
\end{equation}
for some $T_{\ast }>0$ small enough.
\end{theorem} 
Here, we consider problem (\ref{1}) with $p\geq2$ and $f\equiv 0$ with the boundary conditions (\ref{2}) and the initial conditions (\ref{3}).  
We introduce the energy functional $E(t)$ associated with our problem
\begin{equation}
\begin{aligned}
&E(t)= \frac{1}{2} \int_{\Omega} \vert y'\vert^2  dx + \frac{1}{2}   \int_{\Omega} \vert y'_x\vert^2  dx+J(t).
\end{aligned}  \label{Energy}
\end{equation}
where
\begin{eqnarray}
J(t)&=&\frac{1}{2}\int_{\Omega}\Big(1- \int_{0}^{\infty}\mu(s)ds\Big)\vert y_{x}\vert^2   dx\nonumber\\
&+&\frac{1}{2} \int_{\Omega}\int_{0}^{\infty}\mu(s)\vert y_x(t)-y_{x}(t-s)\vert ^2ds dx \nonumber\\
&+&\frac{1}{p}\int_{\Omega} \vert y _x \vert ^pdx- \frac{1}{p}\int_{\Omega}\vert y \vert ^pdx\label{J(t)}
\end{eqnarray}
Now, we introduce the stable set as follows (see\cite{FG2013}, \cite{SV2015})
\begin{eqnarray}
W=\{y\in H_{0}^{1}\cap H^2: I(t)>0, J(t)<d\} \cup\{0\}\label{W}
\end{eqnarray}
where
\begin{eqnarray}
I(t)&=&\int_{\Omega} \Big(1-  \int_{0}^{\infty}\mu(s)ds\Big)\vert y_{x}\vert^2   dx\nonumber\\
&+& \int_{\Omega}\int_{0}^{\infty}\mu(s)\vert y_x(t)-y_{x}(t-s)\vert ^2ds dx \nonumber\\
&+& \int_{\Omega} \vert y_x \vert ^pdx- \int_{0}^{L} \vert y \vert ^pdx \label{I(t)}
\end{eqnarray}
We notice that the mountain pass level $d$ given in (\ref{W}) defined by
\begin{eqnarray}
d &=&\inf\{\sup_{y \in H_{0}^{1}\cap H^2 \backslash \{0\}, \nu \geq
	0}J\left( \nu y \right)\},
\label{d_constantch61}
\end{eqnarray}
Also, by introducing the so called "\textit{Nehari manifold}"
\begin{eqnarray}
\mathcal{N}=\left\{y\in  H_{0}^{1}\cap H^2 \backslash \left\{ 0 \right\} :I\left( t\right)
=0\right\}\nonumber
\end{eqnarray}
It is readily seen that the potential depth $d$ is also characterized by
\begin{eqnarray}
d=\inf_{y\in \mathcal{N}}J(t).
\end{eqnarray}
This characterization of $d$ shows that
\begin{eqnarray}
dist\left(0 ,\mathcal{N}\right) =\min_{
	y \in \mathcal{N}}\left\Vert y \right\Vert _{H_{0}^{1}\cap H^2}
\end{eqnarray}
It is note hard to see this Lemma.
\begin{lemma} \label{lemE'}
	Suppose that (Hyp1) holds. Let $y$ be solution of our equation.  Then the energy functional (\ref{Energy}) is a non-increasing function, i.e., for all $t\geq 0, \nu>0$,
	\begin{eqnarray}
	\frac{d}{dt}E(t)&\leq&- \int_{\Omega} \vert y'_x\vert^2  dx+\frac{1}{2} \int_{\Omega}\int_{0}^{\infty}\mu'(s)\vert y_x(t)-y_{x}(t-s)\vert ^2ds dx  \label{E'}
	\end{eqnarray}
\end{lemma}
\begin{proof} Multiplying (\ref{1}) with $p\geq2$ and $f\equiv 0$ by
	$y'(x,t)$ and integrating over $\Omega$, using Lemma \ref{lemg} we obtain th result.
\end{proof}
As in \cite{IVLove2018}, we will prove the invariance of the set $W.$ That is if for some $t_{0}>0$
if $y(t_0)\in W,$ then $y(t)\in W,$ $\forall t\geq t_{0}$, beginning by the existence of the potential depth in the next Lemma.
\begin{lemma} [\cite{IVLove2018}, Lemma 3.2]
	$d$ is positive constant.
\end{lemma} 
\begin{lemma} [\cite{IVLove2018}, Lemma 3.3]
	 $W$ is a bounded neighborhood of $0$ in $H_{0}^{1}\cap H^2$.
\end{lemma} 
Now, we will show that our local solution $y$ is global in time, for this purpose it suffices to
prove that the norm of the solution is bounded, independently of $t$, this is equivalent to prove the
following theorem.
\begin{theorem} [\cite{IVLove2018}, Theorem 3.4]
	 Suppose that (Hyp1) and
	\begin{eqnarray}
	C^{p}l^{(1-p)}\left( \frac{2p}{p-2}E(0)\right)^{(p-2)}<l.\label{global condition}
	\end{eqnarray}
	hold, where $C$ is the best Poincare's constant. If $y_0(0)\in W,$ $y_1\in H^1_0$, then the solution $y\in W,$ $\forall t\geq 0$.
\end{theorem} 
The next Theorem shows that our local solution is global in time.
\begin{theorem}\label{global theorem} Suppose that (Hyp1), $p\geq
	2$ and (\ref{global condition}) hold. If $y_{0}(0)\in W,$ $y_{1}\in H^1_0$. Then the local solution $y$ is global in time such that $y\in G_{T}$ where
	\begin{equation}
	G_{T}=\left\{
	\begin{array}{ll}
	y: & y\in L^{\infty}\left(\mathbb{R}^{+} ;H^1_0\cap H^2\right) , \\
	& y '\in L^{\infty}\left(\mathbb{R}^{+} ;H^1_0\right)
	\end{array}
	\right\} .
	\end{equation}
\end{theorem}
\begin{proof}
Now, it suffices to show that the following norm
	\begin{equation}
	\int_{\Omega} \vert y'\vert^2   dx+\int_{\Omega}\vert y_{x}\vert^2   dx,
	\end{equation}%
	is bounded independently of $t$.
	
	To achieve this, we use (\ref{Energy}), (\ref{J(t)}) and (\ref{E'}) to get
	\begin{eqnarray}
	E(0) &\geq & E(t)=J(t)+\frac{1}{2}\int_{\Omega}\vert y'\vert^2   dx
	\nonumber \\
	&\geq& \Big(\frac{p-2}{2p}\Big)\Big[\int_{\Omega} \Big(1-  \int_{0}^{\infty}\mu(s)ds\Big)\vert y_{x}\vert^2   dx+ \int_{\Omega}\int_{0}^{\infty}\mu(s)\vert y_x(t)-y_{x}(t-s)\vert ^2ds dx\Big] \nonumber \\
	&+&\frac{1}{2}\int_{\Omega}\vert y'\vert^2   dx+\frac{1}{p}I(t)  \nonumber\\
	&\geq&  \Big(\frac{p-2}{2p}\Big)\Big[l\int_{\Omega} \vert y_{x}\vert^2   dx+ \int_{\Omega}\int_{0}^{\infty}\mu(s)\vert y_x(t)-y_{x}(t-s)\vert ^2ds dx\Big]+\frac{1}{2}\int_{\Omega}\vert y'\vert^2   dx+\frac{1}{p}I(t)  \nonumber\\
	&\geq&  \Big(\frac{l(p-2)}{2p}\Big) \int_{\Omega} \vert y_{x}\vert^2   dx  +\frac{1}{2}\int_{\Omega}\vert y'\vert^2   dx     \nonumber
	\end{eqnarray}%
	since $I(t)$ and $ \int_{\Omega}\int_{0}^{\infty}\mu(s)\vert y_x(t)-y_{x}(t-s)\vert ^2ds dx$ are positive, hence
	\begin{eqnarray}
	\Big(\frac{l(p-2)}{2p}\Big) \int_{\Omega} \vert y_{x}\vert^2   dx  +\frac{1}{2}\int_{\Omega}\vert y'\vert^2   dx\leq CE(0), \nonumber
	\end{eqnarray}%
	where $C$ is a positive constant depending only on $p$ and $l$. This completes the proof.
\end{proof}
\section{General decay rate}
\begin{theorem}\label{decay theorem} Suppose that (Hyp1), $p\geq2$ and (\ref{global condition}) hold. If $y_{0}(0)\in H^1_0\cap H^2,$ $y_{1}\in H^1_0$. Then the energy function (\ref{Energy}) satisfies
	\begin{equation}
	E(t)\leq \kappa_1H_1^{-1}(\kappa t+\kappa_0), \qquad \forall t \geq 0,
	\end{equation}
	where
	\begin{eqnarray}
	 H_1(\tau)=\int_{\tau}^{1} (s H'(\kappa s))^{-1}d s, \qquad \kappa_0, \kappa_1 , \kappa>0.
	\end{eqnarray}
\end{theorem}
In order to prove the main Theorem \ref{decay theorem}, we need to introduce a several Lemmas. To this end, let us introduce the functionals
\begin{eqnarray}
\varphi(t)= \int_{\Omega} y y'dx+\frac{1}{2}\int_{\Omega} \vert y_x\vert ^2dx+ \int_{\Omega} y_x  y'_{x}dx \label{varphi}
\end{eqnarray}
and 
\begin{eqnarray}
\xi(t)&=& -\int_{\Omega} y'\int_{0}^{\infty}\mu(s) [y(t)-y(t-s)]dsdx \nonumber\\
&-& \int_{0}^{t}\int_{\Omega} y'_x\int_{0}^{\infty}\mu(s) [y_x(t)-y_x(t-s)]dsdxd\tau \nonumber\\
&-& \int_{0}^{t}\int_{\Omega}  y''_x\int_{0}^{\infty}\mu(s) [y_x(t)-y_x(t-s)]dsdx d\tau, \label{xi}
\end{eqnarray}
\begin{lemma}\label{L1}
	 Assume that (Hyp1), $p\geq2$ and (\ref{global condition}) hold. Then the functional $\varphi(t)$ introduced in (\ref{varphi}) satisfies, along the solution, the estimate
	 \begin{eqnarray}
	 &&	\varphi' (t)\leq\int_{\Omega} y'^2dx+ \int_{\Omega}y'^2_xdx \nonumber\\
	 &&-\Big[\frac{l}{2}  -\Big(\frac{2p}{(p-2)l}E(0)\Big)^{(p-2)/2}\Big] \int_{\Omega} \vert y_x\vert ^2dx +\frac{ (1-l)}{2l}\int_{\Omega} \int_{0}^{\infty}\mu(s)\vert y_{x}(t-s)- y_x \vert^2 ds dx \nonumber
	 \end{eqnarray} 
\end{lemma}
\begin{proof}
	\begin{eqnarray}
&&	\varphi' (t)=\int_{\Omega} y'^2dx+ \int_{\Omega} y'^2_xdx- \int_{\Omega}y_x^2dx+\int_{\Omega}\vert y \vert^{p}dx -\int_{\Omega}\vert y_x \vert^{p}dx\nonumber\\
	&&+ \int_{\Omega} y_x\int_{0}^{\infty}\mu(s)y_{x}(t-s)dsdx \nonumber
	\end{eqnarray}
	The next term can be treated as follows
	\begin{eqnarray}
	 \int_{\Omega} y_x\int_{0}^{\infty}\mu(s)y_{x}(t-s)dsdx\leq \frac{1}{2}\int_{\Omega} y_x^2dx+\frac{1}{2} \int_{\Omega}\Big( \int_{0}^{\infty}\mu(s)y_{x}(t-s)ds\Big)^2dx\nonumber\\
	 \qquad\qquad \leq
	 \frac{1}{2}\int_{\Omega}y_x^2dx+\frac{1}{2} \int_{\Omega}\Big( \int_{0}^{\infty}\mu(s)\vert y_{x}(t-s)-  y_x \vert  + \vert y_x \vert ds\Big)^2dx\nonumber
	\end{eqnarray}
	Thanks to Cauchy-Schwarz inequality, Young's inequality to obtain, for some $\nu >0$,
	\begin{eqnarray}
 && \int_{\Omega}\Big( \int_{0}^{\infty}\mu(s)\vert y_{x}(t-s)-  y_x \vert  + \vert y_x \vert ds\Big)^2dx\nonumber\\
  && \leq  \int_{\Omega}\Big( \int_{0}^{\infty}\mu(s)\vert y_{x}(t-s)-  y_x \vert ds \Big)^2dx +\int_{\Omega}\Big( \int_{0}^{\infty}\mu(s)  \vert y_x \vert ds\Big)^2dx\nonumber\\
  && +2\int_{\Omega}\Big( \int_{0}^{\infty}\mu(s)\vert y_{x}(t-s)-  y_x \vert ds \Big)  \Big( \int_{0}^{\infty}\mu(s)  \vert y_x \vert ds\Big) dx\nonumber\\
  &&\leq \Big(1+\frac{1}{\nu}\Big)\int_{\Omega}\Big( \int_{0}^{\infty}\mu(s)\vert y_{x}(t-s)-  y_x \vert ds\Big)^2dx \nonumber\\
  &+& (1+\nu) \int_{\Omega}\Big( \int_{0}^{\infty}\mu(s)ds  \vert y_x \vert \Big)^2dx\nonumber\\
  &\leq & \Big(1 +\frac{1}{\nu}\Big)(1-l)\int_{\Omega}  \int_{0}^{\infty}\mu(s)\vert y_{x}(t-s)-  y_x \vert^2 ds dx+(1+\nu) (1-l)^2\int_{\Omega}\vert y_x \vert ^2dx. \nonumber
	\end{eqnarray}
	Then
	\begin{eqnarray}
	&&	\varphi' (t)\leq\int_{\Omega}y'^2dx+ \int_{\Omega}y'^2_xdx \nonumber\\
	&&+\Big[\frac{1}{2}+\frac{1}{2}(1+\nu) (1-l)^2-1\Big]\int_{\Omega} \vert y_x\vert ^2dx\nonumber\\
	&&+\frac{1}{2}\Big(1 +\frac{1}{\nu}\Big)(1-l)\int_{\Omega}  \int_{0}^{\infty}\mu(s)\vert y_{x}(t-s)-  y_x \vert^2 ds dx \nonumber\\
	&&+\int_{\Omega}\vert y \vert^{p}dx.  
	\end{eqnarray}
	By the continuous embedding for $p\geq 2$, we have
	\begin{eqnarray}
	&&\int_{\Omega}\vert y \vert^{p}dx \leq c^p\Big(\int_{\Omega}\vert y_x \vert^{2}dx\Big)^{p/2}\nonumber\\
	&& \leq c^p\Big(\int_{\Omega}\vert y_x \vert^{2}dx\Big)^{(p-2)/2} \int_{\Omega}\vert y_x \vert^{2}dx  \nonumber\\
	&& \leq c^p \Big(\frac{2p}{(p-2)l}E(0)\Big)^{(p-2)/2}\int_{\Omega}\vert y_x \vert^{2}dx \label{p}
	\end{eqnarray}
	Since (\ref{global condition}), we can choose $\nu = \frac{l}{1-l}$ to get 
	\begin{eqnarray}
	&&	\varphi' (t)\leq\int_{\Omega}y^2dx+ \int_{\Omega} y'^2_xdx \nonumber\\
	&&-\Big[\frac{l}{2}  -c^pl^{1-p}\Big(\frac{2p}{(p-2)}E(0)\Big)^{(p-2)}\Big]\int_{\Omega} \vert y_x\vert ^2dx\nonumber\\
	&&+\frac{ (1-l)}{2l}\int_{\Omega} \int_{0}^{\infty}\mu(s)\vert y_{x}(t-s)-  y_x \vert^2 ds dx \nonumber
	\end{eqnarray}
\end{proof}
\begin{lemma}\label{L2}
Assume that (Hyp1), $p\geq2$ hold. Then the functional  introduced in (\ref{xi}) satisfies, along the solution, the estimate
Then, for $\nu< (1-l)$
\begin{eqnarray}
\xi'(t)&\leq & -a   \int_{\Omega}\vert y_x\vert ^2 dx-((1-l)-\nu) \int_{\Omega}y'^2dx\nonumber\\
&+&b\int_{\Omega} \int_{0}^{\infty}\mu(s) \vert y_x(t)-y_x(t-s)\vert^2dsdx \nonumber\\  
&+& \frac{c\mu(0)}{4\nu}\int_{\Omega} \int_{0}^{\infty}(-\mu'(s)) \vert y(t)-y(t-s)\vert
^2dsdx.\nonumber
\end{eqnarray} 
where  
$$
a=c(\nu)\Big(1+2(1-l)^2- \Big(\frac{2p}{(p-2)l}E(0)\Big)^{(p-2)/2} \Big)>0
$$
and for all $\nu>0$ 
$$
b=\frac{1-l}{4\nu}+(2\nu+\frac{1}{4\nu})(1-l)+2\nu (1-l)^{p-1}c \Big( \frac{8}{1-l}E(0)+2m_0^2\Big)^{(p-2/2)}>0.
$$  
\end{lemma}
\begin{proof}
	We have 
\begin{eqnarray}
\xi'(t)&= & \int_{\Omega} y_x\int_{0}^{\infty}\mu(s) [y_x(t)-y_x(t-s)]dsdx \nonumber\\
&-&  \int_{\Omega}\int_{0}^{\infty}\mu(s)y_{x}(t-s)ds \int_{0}^{\infty}\mu(s) [y_x(t)-y_x(t-s)]dsdx\nonumber\\
&-& \int_{\Omega}\vert y \vert^{p-2}y  \int_{0}^{\infty}\mu(s) [u(t)-u(t-s)]dsdx\nonumber\\
&+& \int_{\Omega}\vert y_x \vert^{p-2}y_x   \int_{0}^{\infty}\mu(s) [y_x(t)-y_x(t-s)]dsdx\nonumber\\
&-&\int_{\Omega}y'\int_{0}^{\infty}\mu'(s) [y(t)-y(t-s)]dsdx- (1-l) \int_{\Omega}y'^2dx.\nonumber
\end{eqnarray}
For $\nu >0$, we have
\begin{eqnarray}
&&\int_{\Omega} \ y_x\int_{0}^{\infty}\mu(s) [y_x(t)-y_x(t-s)]dsdx\nonumber\\
&&\leq \nu \int_{\Omega}\vert y_x\vert ^2 dx+\frac{1-l}{4\nu}\int_{\Omega} \int_{0}^{\infty}\mu(s) \vert y_x(t)-y_x(t-s)\vert^2dsdx,\nonumber
\end{eqnarray}
and
\begin{eqnarray}
&&\int_{\Omega}\int_{0}^{\infty}\mu(s)y_{x}(t-s)ds  \int_{0}^{\infty}\mu(s) [y_x(t)-y_x(t-s)]dsdx\nonumber\\
&&\leq 2\nu(1-l)^2 \int_{\Omega} \vert y_x\vert ^2 dx+(2\nu+\frac{1}{4\nu})(1-l)\int_{\Omega} \int_{0}^{\infty}\mu(s) \vert y_x(t)-y_x(t-s)\vert^2dsdx,\nonumber
\end{eqnarray}
we have the estimate
\begin{eqnarray}
&&\int_{\Omega}  \vert y(t)-y(t-s)\vert ^2 dx\nonumber\\
&&\leq 2\int_{\Omega} \vert y_x(t)\vert ^2
dx+2\int_{\Omega}\vert y_x(t-s)\vert^2 dx    \nonumber\\
&&\leq  4\sup_{s>0}\int_{\Omega} \vert y_x(t)\vert ^2
dx+2\sup_{\tau<0}\int_{\Omega} \vert y_x(\tau)\vert^2 dx \nonumber\\
&&\leq    \sup_{s>0}\int_{\Omega}  \vert y_x(t)\vert ^2
dx+2\sup_{\tau<0}\int_{\Omega} \vert y_{0x}(\tau)\vert^2 dx \nonumber\\
&&\leq     \frac{8}{1-l}E(0)+2m_0^2,  \label{N1}
\end{eqnarray}
where $\int_{\Omega} \vert y_{0x}(\tau)\vert^2 dx\leq m_0, m_0 \geq 0,\quad \forall \tau >0$. Since $p\geq2$, we have by using (\ref{p}) and the previous estimate
\begin{eqnarray}
&&\int_{\Omega}\vert y \vert^{p-2}y \int_{0}^{\infty}\mu(s) [y(t)-y(t-s)]dsdx\nonumber\\
&&\leq \nu \int_{\Omega}\Big \vert  \int_{0}^{\infty}\mu(s) \vert y(t)-y(t-s)\vert ds\Big\vert^pdx+c(\nu)\int_{\Omega}\vert y \vert^{p}dx\nonumber\\
&&\leq \nu (1-l)^{p-1}\int_{\Omega}    \int_{0}^{\infty}\mu(s) \vert y(t)-y(t-s)\vert^p ds dx+c(\nu)\int_{\Omega}\vert y \vert^{p}dx\nonumber\\
&&\leq\nu (1-l)^{p-1}c  \int_{0}^{\infty}\mu(s)\Big (\int_{\Omega}  \vert y_x(t)-y_x(t-s)\vert^2 dx\Big)^{p/2} ds+c(\nu)\int_{\Omega}\vert y \vert^{p}dx\nonumber\\
&& \leq \nu (1-l)^{p-1}c \Big( \frac{8}{1-l}E(0)+2m_0^2\Big)^{(p-2/2)} \int_{0}^{\infty}\mu(s) \int_{\Omega}  \vert y_x(t)-y_x(t-s)\vert^2 dx\nonumber\\
&&+c(\nu)\int_{\Omega}\vert y \vert^{p}dx\nonumber\\
&&\leq  \nu (1-l)^{p-1}c \Big( \frac{8}{1-l}E(0)+2m_0^2\Big)^{(p-2/2)} \int_{0}^{\infty}\mu(s) \int_{\Omega}  \vert y_x(t)-y_x(t-s)\vert^2 dx\nonumber\\
&&+c(\nu) \Big(\frac{2p}{(p-2)l}E(0)\Big)^{(p-2)/2}\int_{\Omega}\vert y_x \vert^{2}dx. \nonumber
\end{eqnarray}
Similarely to estimate
\begin{eqnarray}
&&\int_{\Omega}\vert y_x \vert^{p-2}y_x \int_{0}^{\infty}\mu(s) [y_x(t)-y_x(t-s)]dsdx\nonumber\\
&&\leq \nu \int_{\Omega} \Big \vert  \int_{0}^{\infty}\mu(s) \vert y_x(t)-y_x(t-s)\vert ds\Big\vert^pdx+c(\nu)\int_{\Omega}\vert y_x \vert^{p}dx\nonumber\\
&&\leq \nu (1-l)^{p-1}\int_{\Omega}  \int_{0}^{\infty}\mu(s) \vert y_x(t)-y_x(t-s)\vert^p ds dx+c(\nu)\int_{\Omega}\vert y_x \vert^{p}dx\nonumber\\
&&\leq\nu (1-l)^{p-1}c  \int_{0}^{\infty}\mu(s)\Big (\int_{\Omega}\vert y_x(t)-y_x(t-s)\vert^2 dx\Big)^{p/2} ds+c(\nu)\int_{\Omega}\vert y_x \vert^{p}dx\nonumber\\
&& \leq \nu (1-l)^{p-1}c \Big( \frac{8}{1-l}E(0)+2m_0^2\Big)^{(p-2/2)} \int_{0}^{\infty}\mu(s) \int_{\Omega}\vert y_x(t)-y_x(t-s)\vert^2 dx\nonumber\\
&&+c(\nu)\int_{\Omega}\vert y _x\vert^{2}dx.\nonumber
\end{eqnarray}
The nest term estimated as
\begin{eqnarray}
&&-\int_{\Omega}y'\int_{0}^{\infty}\mu'(s) [y(t)-y(t-s)]dsdx\nonumber\\
&& \leq \nu \int_{\Omega}\vert y'\vert ^2dx+\frac{c\mu(0)}{4\nu}\int_{\Omega} \int_{0}^{\infty}(-\mu'(s)) \vert y(t)-y(t-s)\vert
^2dsdx.\nonumber
\end{eqnarray}
A combination of all estimates gives 
\begin{eqnarray}
\xi'(t)&\leq & \nu \int_{\Omega}\vert y_x\vert ^2 dx+\frac{1-l}{4\nu}\int_{\Omega} \int_{0}^{\infty}\mu(s) \vert y_x(t)-y_x(t-s)\vert^2dsdx \nonumber\\
&+&  2\nu(1-l)^2 \int_{\Omega} \vert y_x\vert ^2 dx+(2\nu+\frac{1}{4\nu})(1-l)\int_{\Omega} \int_{0}^{\infty}\mu(s) \vert y_x(t)-y_x(t-s)\vert^2dsdx\nonumber\\
&+& \nu (1-l)^{p-1}c \Big( \frac{8}{1-l}E(0)+2m_0^2\Big)^{(p-2/2)} \int_{0}^{\infty}\mu(s) \int_{\Omega}  \vert y_x(t)-y_x(t-s)\vert^2 dx\nonumber\\
&&+c(\nu) \Big(\frac{2p}{(p-2)l}E(0)\Big)^{(p-2)/2}\int_{\Omega}\vert y_x \vert^{2}dx \nonumber\\
&+&  \nu (1-l)^{p-1}c \Big( \frac{8}{1-l}E(0)+2m_0^2\Big)^{(p-2/2)} \int_{0}^{\infty}\mu(s) \int_{\Omega}\vert y_x(t)-y_x(t-s)\vert^2 dx\nonumber\\
&&+c(\nu)\int_{\Omega}\vert y _x\vert^{2}dx\nonumber\\
&+&\nu \int_{\Omega}\vert y'\vert ^2dx+\frac{c\mu(0)}{4\nu}\int_{\Omega} \int_{0}^{\infty}(-\mu'(s)) \vert y(t)-y(t-s)\vert
^2dsdx- (1-l) \int_{\Omega}y'^2dx.\nonumber
\end{eqnarray} 
Then, for $\nu< (1-l)$
\begin{eqnarray}
\xi'(t)&\leq & -a   \int_{\Omega}\vert y_x\vert ^2 dx-((1-l)-\nu) \int_{\Omega}y'^2dx\nonumber\\
&+&b\int_{\Omega} \int_{0}^{\infty}\mu(s) \vert y_x(t)-y_x(t-s)\vert^2dsdx \nonumber\\  
&+& \frac{c\mu(0)}{4\nu}\int_{\Omega} \int_{0}^{\infty}(-\mu'(s)) \vert y(t)-y(t-s)\vert
^2dsdx.\nonumber
\end{eqnarray} 
where by (\ref{global condition}), we have
$$
a=c(\nu)\Big(1+2(1-l)^2- \Big(\frac{2p}{(p-2)l}E(0)\Big)^{(p-2)/2} \Big)>0
$$
and for all $\nu>0$, we have
$$
b=\frac{1-l}{4\nu}+(2\nu+\frac{1}{4\nu})(1-l)+2\nu (1-l)^{p-1}c \Big( \frac{8}{1-l}E(0)+2m_0^2\Big)^{(p-2/2)}>0.
$$
\end{proof}
Let define the functional
\begin{eqnarray}
L(t)=\varepsilon_1E(t)+ \varphi(t) +\varepsilon_2 \xi(t), \quad \varepsilon_1, \varepsilon_2>0 \label{L(t)}
\end{eqnarray}
We need the next lemma, which means that there is equivalent between
the Lyapunov and energy functions
\begin{lemma} \label{L3}
For $\varepsilon_1, \varepsilon_2>1$, we have
\begin{eqnarray}
L \sim E,
\end{eqnarray}
\end{lemma}
\begin{proof}
	By (\ref{L(t)}) we have
	\begin{eqnarray}
	\vert L(t)-\xi_{1}E(t)\vert &\leq& \vert\varphi(t)\vert+\varepsilon_{2}\vert\xi_{2}(t)\vert \nonumber\\
	&\leq&\int_{\Omega} \vert y y'\vert dx+\frac{1}{2}\int_{\Omega} \vert y_x\vert ^2dx+ \int_{\Omega} \vert y_x  y'_{x}\vert dx\nonumber\\
	&+&  \int_{\Omega} \Big\vert  y'\int_{0}^{\infty}\mu(s) [y(t)-y(t-s)]ds\Big\vert dx \nonumber\\
	&+& \int_{0}^{t}\int_{\Omega} \Big\vert y'_x\int_{0}^{\infty}\mu(s) [y_x(t)-y_x(t-s)]dsd\tau \Big\vert dx\nonumber\\
	&+& \int_{0}^{t}\int_{\Omega}\Big\vert  y''_x\int_{0}^{\infty}\mu(s) [y_x(t)-y_x(t-s)]ds d\tau\Big\vert dx,\nonumber
	\end{eqnarray}
	Thanks to Holder and Young's inequalities, we have
	\begin{eqnarray}
	\int_{\Omega} \vert y y'\vert dx &\leq& \left(\int_{\Omega} \vert y \vert^2 dx\right)^{1/2}\left(\int_{\Omega}\vert  y'\vert^2 dx\right)^{1/2}\nonumber\\
	&\leq& \frac{1}{2}\left(\int_{\Omega} \vert y \vert^2 dx\right)+\frac{1}{2}\left(\int_{\Omega} \vert  y'\vert^2 dx\right)\nonumber\\
	&\leq& \frac{c}{2}\left(\int_{\Omega}\vert y_x \vert^2 dx\right)+\frac{1}{2}\left(\int_{\Omega}\vert  y'\vert^2 dx\right)\nonumber
	\end{eqnarray}
	Similarly
	\begin{eqnarray}
	\int_{\Omega} \vert y_x y'_x\vert dx
	\leq \frac{1}{2}\left(\int_{\Omega} \vert y _x \vert^2 dx\right)+\frac{1}{2}\left(\int_{\Omega}\vert  y'_x \vert^2 dx\right)\nonumber
	\end{eqnarray}
	and
	\begin{eqnarray}
	&& \int_{\Omega} \Big\vert  y'\int_{0}^{\infty}\mu(s) [y(t)-y(t-s)]ds\Big\vert dx \nonumber\\
	&&\leq \frac{1}{2}\left(\int_{\Omega} \vert  y'\vert^2 dx\right)+ \frac{c}{2}\int_{\Omega}  \int_{0}^{\infty}\mu(s)  \vert y_x(t)-y_x(t-s) \vert ^2dsdx\nonumber
	\end{eqnarray}
	and
	\begin{eqnarray}
	&&  \int_{0}^{t}\int_{\Omega}\Big\vert y'_x\int_{0}^{\infty}\mu(s) [y_x(t)-y_x(t-s)]dsd\tau \Big\vert dx\nonumber\\
	&&\leq \frac{1}{2}\int_{0}^{t}\left(\int_{\Omega} \vert  y'_x\vert^2 dx\right)d\tau+ \frac{1}{2}\int_{0}^{t}\int_{\Omega} \int_{0}^{\infty}\mu(s)  \vert y_x(t)-y_x(t-s) \vert ^2dsdxd\tau\nonumber
	\end{eqnarray}
	and
	\begin{eqnarray}
	&&  \int_{0}^{t}\int_{\Omega} \Big\vert  y''_x\int_{0}^{\infty}\mu(s) [y_x(t)-y_x(t-s)]ds d\tau\Big\vert dx\nonumber\\
	&&\leq \frac{1}{2}\int_{0}^{t}\left(\int_{\Omega} \vert  y''_x\vert^2 dx\right)d\tau+ \frac{1}{2}\int_{0}^{t}\int_{\Omega}  \int_{0}^{\infty}\mu(s)  \vert y_x(t)-y_x(t-s) \vert ^2dsdxd\tau\nonumber
	\end{eqnarray}
	Then, for some $c>0$, we have
	\begin{eqnarray}
	\vert L(t)-\varepsilon_{1}E(t)\vert &\leq& c(E(t) \nonumber 
	\end{eqnarray}
	Therefore, we can choose $\varepsilon_{1}$ so that
	\begin{eqnarray}
	L(t)\sim E(t).
	\end{eqnarray}
\end{proof}
\begin{lemma} \label{L4}
Assume (Hyp1) hold. Then there exist strictly positive constants $c$ such that
\begin{eqnarray}
L'(t) \leq - \varepsilon_{1}E(t)+c\int_{\Omega}\int_{0}^{\infty}\mu(s) \vert y_x(t)-y_x(t-s)\vert
^2dsdx \label{L'}
\end{eqnarray}
\end{lemma}
\begin{proof}
By Lemma \ref{lemE'}, Lemma \ref{L1} and Lemma \ref{L2}, we have   
\begin{eqnarray}
L'(t)&=&\varepsilon_1E'(t)+ \varphi'(t)+\varepsilon_2 \xi'(t)\nonumber\\
&\leq &-[\varepsilon_2[(1-l)- \nu]-1] \int_{0}^{L}\vert y'\vert^2dx -((1-\frac{\nu}{2})\varepsilon_1-1)\int_{\Omega}\vert y'_x\vert^2  dx\nonumber\\
&-&\Big[\frac{l}{2}  -\Big(\frac{2p}{(p-2)l}E(0)\Big)^{(p-2)/2}-\varepsilon_{2} a \Big]\int_{\Omega} \vert y_x\vert ^2dx\nonumber\\
&+&\Big[ \frac{ (1-l)}{2l}+b\varepsilon_2\Big]\int_{\Omega} \int_{0}^{\infty}\mu(s)\vert y_{x}(t-s)-  y_x \vert^2 ds dx \nonumber\\ 
&+&\Big[\frac{\varepsilon_1 }{2}-\varepsilon_2\frac{c\mu(0)}{4\nu}\Big]\int_{\Omega} \int_{0}^{\infty}\mu'(s) \vert y_x(t)-y_x(t-s)\vert
^2dsdx,\nonumber
\end{eqnarray} 
where, by (\ref{global condition}), we have
$$
a=c(\nu)\Big(1+2(1-l)^2- \Big(\frac{2p}{(p-2)l}E(0)\Big)^{(p-2)/2} \Big)>0
$$
and for all $\nu>0$ 
$$
b=\frac{1-l}{4\nu}+(2\nu+\frac{1}{4\nu})(1-l)+2\nu (1-l)^{p-1}c \Big( \frac{8}{1-l}E(0)+2m_0^2\Big)^{(p-2/2)}>0.
$$  
Now, we choose $\nu$, and whence this constant is fixed. we can choice $\varepsilon_1, \varepsilon_2$ small enough such that for $p\geq2$ there exist $c>0$ and (\ref{L'}) yields.
\end{proof}
\begin{lemma}\label{L5}
Assume that (Hyp1) hold. Then, there exist $\gamma, \gamma_0>0$ such that for all $t>0$
\begin{eqnarray}
\int_{\Omega}\int_{0}^{\infty}\mu(s) \vert y_x(t)-y_x(t-s)\vert ^2dsdx  \leq \frac{ -\gamma E'(t)}{H'(\gamma_0E(t))}+ \gamma\gamma_0E(t) \label{gamma}
\end{eqnarray}
\end{lemma}
\begin{proof}
	Let $H^{*}$ be the convex conjugate of $g$ in the sense of Young (see \cite{Arnold}, pages 61-64), then
	\begin{eqnarray}
	H^{*}(s)&=&s(H')^{-1}(s)-H[(H')^{-1}(s)]\nonumber\\
	&\leq& s(H')^{-1}(s),\quad s\in (0, H'(r)) \label{g*}
	\end{eqnarray}
	and satisfies the following Young's inequality
	\begin{eqnarray}
	AB\leq H^{*}(A)+H(B),\quad A \in (0,H'(r)), B \in (0,r]. \label{young}
	\end{eqnarray}
	for
$$
A=H^{-1}\Big(-r_2\mu'(s)\int_{\Omega}\vert y_x(t)-y_x(t-s)\vert ^2dx \Big)
$$
$$
B=\frac{r_1H'(\gamma_0E(t))\mu(s)\int_{\Omega}  \vert y_x(t)-y_x(t-s)\vert ^2dx}{H^{-1}\Big(-r_2\mu'(s)\int_{\Omega} \vert y_x(t)-y_x(t-s)\vert ^2 dx \Big)}.
$$
Then, for $r_1, r_2>0$, we have (see \cite{IV2018})
\begin{eqnarray}
&&\int_{\Omega}\int_{0}^{\infty}\mu(s) \vert y_x(t)-y_x(t-s)\vert ^2dsdx\nonumber\\
&&=\frac{1}{r_1H'(\gamma_0E(t)) }\int_{0}^{\infty}\Big\{H^{-1}\Big(-r_2\mu'(s)\int_{\Omega} \vert y_x(t)-y_x(t-s)\vert ^2 dx\Big)   \nonumber\\
&& \times \qquad \frac{r_1H'(\gamma_0E(t))\mu(s)\int_{\Omega} \vert y_x(t)-y_x(t-s)\vert ^2 dx}{H^{-1}\Big(-r_2\mu'(s)\int_{\Omega}  \vert y_x(t)-y_x(t-s)\vert ^2 dx \Big)} \Big\}ds\nonumber\\
&&\leq -\frac{r_2}{r_1H'(\gamma_0E(t))}\int_{\Omega}\int_{0}^{\infty}\mu'(s) \vert y_x(t)-y_x(t-s)\vert ^2dsdx\nonumber\\
&&+\frac{1}{r_1H'(\gamma_0E(t))}\int_{0}^{\infty}H^*\Big( \frac{r_1H'(\gamma_0E(t))\mu(s)\int_{\Omega} \vert y_x(t)-y_x(t-s)\vert ^2 dx}{H^{-1}\Big(-r_2\mu'(s)\int_{\Omega} \vert y_x(t)-y_x(t-s)\vert ^2 dx \Big)} \Big)ds \nonumber
\end{eqnarray}
By (\ref{E'}), (\ref{g*}), we have
\begin{eqnarray}
&&\int_{\Omega}\int_{0}^{\infty}\mu(s) \vert y_x(t)-y_x(t-s)\vert ^2dsdx \leq \frac{2r_2}{r_1H'(\gamma_0E(t)) }E'(t)   \nonumber\\
&& +   \int_{0}^{\infty}\Big\{\frac{ \mu(s)\int_{\Omega}  \vert y_x(t)-y_x(t-s)\vert ^2 dx}{H^{-1}\Big(-r_2\mu'(s)\int_{\Omega} \vert y_x(t)-y_x(t-s)\vert ^2 dx \Big)} \nonumber\\
&&\times  H'^{-1}\Big( \frac{r_1H'(\gamma_0E(t))\mu(s)\int_{\Omega} \vert y_x(t)-y_x(t-s)\vert ^2 dx}{H^{-1}\Big(-r_2\mu'(s)\int_{\Omega}  \vert y_x(t)-y_x(t-s)\vert ^2 dx \Big)} \Big)\Big\}ds \nonumber
\end{eqnarray}
By the fact that $H^{-1}$ is concave and $H^{-1}(0)=0$, the function $h(s)=\frac{s}{H^{-1}(s)}$, such that for $0\leq s_1< s_2$, we have
$$
h(s_1)\leq h(s-2).
$$
Therefore, using (\ref{N1}) to get
\begin{eqnarray}
 &&H'^{-1}\Big( \frac{r_1H'(\gamma_0E(t))\mu(s)\int_{\Omega}  \vert y_x(t)-y_x(t-s)\vert ^2 dx}{H^{-1}\Big(-r_2\mu'(s)\int_{\Omega} \vert y_x(t)-y_x(t-s)\vert ^2 dx \Big)} \Big)\nonumber\\
 && =H'^{-1}\Big[\frac{r_1H'(\gamma_0E(t))\mu(s)}{-r_2\mu'(s)}h\Big(    -r_2\mu'(s)\int_{\Omega} \vert y_x(t)-y_x(t-s)\vert ^2 dx \Big)  \Big]\nonumber\\
 &&\leq H'^{-1}\Big[\frac{r_1H'(\gamma_0E(t))\mu(s)}{-r_2\mu'(s)}h\Big(    -r_2\mu'(s)\Big(\frac{8}{1-l}E(0)+2m_0^2\Big) \Big)  \Big]\nonumber\\
 && \leq  H'^{-1}\Big[\frac{r_1\Big(\frac{8}{1-l}E(0)+2m_0^2\Big)H'(\gamma_0E(t))\mu(s)}{H^{-1}\Big(    -r_2\mu'(s)\Big(\frac{8}{1-l}E(0)+2m_0^2\Big) \Big)  } \Big].\nonumber
\end{eqnarray}
Then,
\begin{eqnarray}
&&\int_{\Omega}\int_{0}^{\infty}\mu(s) \vert y_x(t)-y_x(t-s)\vert ^2dsdx\nonumber\\
&& \leq -\frac{2r_2}{r_1H'(\gamma_0E(t)) }E'(t)   \nonumber\\
&& + \Big(\frac{8}{1-l}E(0)+2m_0^2\Big)  \int_{0}^{\infty}\Big\{\frac{ \mu(s) }{H^{-1}\Big(-r_2\Big(\frac{8}{1-l}E(0)+2m_0^2\Big) \mu'(s)  \Big)} \nonumber\\
&&\times  H'^{-1}\Big[\frac{r_1\Big(\frac{8}{1-l}E(0)+2m_0^2\Big)H'(\gamma_0E(t))\mu(s)}{H^{-1}\Big(    -r_2\mu'(s)\Big(\frac{8}{1-l}E(0)+2m_0^2\Big) \Big)  } \Big]\Big\}ds \nonumber
\end{eqnarray}
By (Hyp1), we have
$$
\sup_{s \in \mathbb{R}^+}\frac{ \mu(s) }{H^{-1}\Big(-  \mu'(s)  \Big)}= \kappa_1 <\infty
$$
$$
 \int_{0}^{\infty}\frac{ \mu(s) }{H^{-1}\Big(-  \mu'(s)  \Big)}= \kappa_2 <\infty.
$$
Since $H'^{-1}$ is nondecreasing, we choose $r_1,r_2$ such that
\begin{eqnarray}
&&\int_{\Omega} \int_{0}^{\infty}\mu(s) \vert y_x(t)-y_x(t-s)\vert ^2dsdx\nonumber\\
&& \leq -\frac{2\kappa_2}{ H'(\gamma_0E(t)) }E'(t)   \nonumber + \Big(\frac{8}{1-l}E(0)+2m_0^2\Big) H'^{-1} H'(\gamma_0E(t)) \int_{0}^{\infty} \frac{ \mu(s) }{H^{-1}\Big(- \mu'(s)  \Big)} \nonumber \\
&&\leq -\frac{2\kappa_2}{ H'(\gamma_0E(t)) }E'(t)   \nonumber + \Big(\frac{8}{1-l}E(0)+2m_0^2\Big)\gamma_0E(t).  \nonumber
\end{eqnarray}
This completes the proof.
\end{proof}
\begin{proof} (of Theorem \ref{decay theorem})
	Multiplying (\ref{L'}) by $H'(\gamma_0E(t))$ and using results in (\ref{gamma})
	\begin{eqnarray}
	H'(\gamma_0E(t))L'(t) &\leq& - \varepsilon_{1}H'(\gamma_0E(t))E(t)+cH'(\gamma_0E(t))\int_{\Omega} \int_{0}^{\infty}\mu(s) \vert y_x(t)-y_x(t-s)\vert
	^2dsdx  \nonumber\\
	&\leq& - [\varepsilon_{1}- c\gamma\gamma_0]H'(\gamma_0E(t))E(t)-c\gamma E'(t) .  \nonumber\\
	\end{eqnarray}
	We choose $\gamma_0$ small enough so that $\varepsilon_{1}- c\gamma\gamma_0>0$. \\
	Put
	$$
	g(t)=H'(\gamma_0E(t))L(t)+c\gamma E(t) \sim E(t),
	$$
	then,
	$$
	g'(t)\leq -\kappa g(t)H'(\gamma_0g(t)),
	$$
	which implies that $(H_1(g))'\geq \kappa$, where
	$$
	H_1(\tau)=\int_{\tau}^{1}\frac{1}{sH'(\gamma_0 s)}ds, \quad 0<\tau <1.
	$$
	Integrating $(H_1(g))'\geq \kappa$ over $[0,t]$ we get
	$$
	g(t)\leq H_1^{-1}(\kappa t+ \kappa_0)
	$$
	the equivalence between $E(t)$ and $g(t)$ gives the result.
\end{proof}

\end{document}